\documentclass[a4paper,10pt]{article}
\usepackage{pgf,tikz}
\usetikzlibrary{arrows}
\usepackage{geometry}
\usepackage{latexsym}
\usepackage{amsfonts, amsmath, amssymb, amsthm}
\usepackage[utf8]{inputenc}
\usepackage[T1]{fontenc}
\usepackage[english]{babel}
\usepackage{stmaryrd}
\usepackage{subfig}
\usepackage{frcursive}
\usepackage{mathrsfs}
\usepackage{enumitem}
\usepackage{bbm}
\usepackage{palatino}

\newcommand{\N}{\mathbb{N}}
\newcommand{\Z}{\mathbb{Z}}

\newcommand{\R}{\mathbb{R}}
\newcommand{\C}{\mathbb{C}}

\newcommand{\hA}{\hat{A}}
\newcommand{\bin}[1]{\underline{#1}_{_2}}

\newtheorem{theorem}{Theorem}[section]
\newtheorem{lemma}[theorem]{Lemma}
\newtheorem{proposition}[theorem]{Proposition}

\theoremstyle{definition}
\newtheorem{remark}[theorem]{Remark}
\newtheorem{definition}[theorem]{Definition}

\title{Normal distribution of correlation measures of binary sum-of-digits functions}
\author{Jordan Emme\thanks{Laboratoire de Mathématiques d’Orsay, Univ. Paris-Sud, CNRS, Université Paris-Saclay, 91405 Orsay, France} 
\and
 Pascal Hubert \thanks{Aix-Marseille Université, CNRS, Centrale Marseille, I2M, UMR 7373, 13453 Marseille,
France}}
\date{}
\begin{document}

\maketitle

\begin{abstract}
In this paper we study correlation measures introduced in \cite{emme_asymptotic_2017}. Denote by $\mu_a(d)$ the asymptotic density of the set $\mathcal{E}_{a,d}=\{n \in \N, \ s_2(n+a)-s_2(n)=d\}$ (where $s_2$ is the sum-of-digits function in base 2). Then, for any point $X$ in $\{0,1\}^\N$, define the integer sequence $\left(a_X (n)\right)_{n\in \N}$ such that the binary decomposition of $a_X (n) $ is the prefix of length $n$ of $X$. We prove that for \textit{any} shift-invariant ergodic probability measure $\nu$ on $\{0,1\}^\N$, the sequence $\left(\mu_{a_X(n)}\right)_{n \in \N}$ satisfies a central limit theorem. This result was proven in the case where $\nu$ is the symmetric Bernoulli measure in \cite{emme_central_2018}.
\end{abstract}

\tableofcontents

\section{Introduction}
\subsection{Background}

In this paper, we are interested in the asymptotic behaviour of a certain family of probability measures defined via sum-of-digits functions. We introduce these measures quickly in what follows. 

For any non-negative integer $n$, there exists a tuple $(n_0,...,n_{k-1})$ in $\{0,1\}^k$ such that:
$$
n = \sum_{j=0}^{k-1}n_j2^j.
$$
We define the sum-of-digits function $s_2:\N \rightarrow \N$ by
$$
s_2(n)=\sum_{j=0}^{k-1}n_j
$$ 
where $\N$ denotes the set of non-negative integers.
We are interested in the following family of sets
$$
\forall (a,d)\in \N\times\Z,\quad \mathcal{E}_{a,d}:=\left\{n\in \N ,\ s_2(n+a)-s_2(n)=d \right\} 
$$
and we know from \cite{besineau_independance_1973} and \cite{emme_asymptotic_2017} that these sets admit asymptotic densities. So let us define, for every $a$ in $\N$, the probability measure $\mu_a$.

\begin{definition}
$$
\forall (a,d)\in \N\times\Z,\quad \mu_a(d):=\lim_{N\rightarrow+\infty}\frac1N \left|  \mathcal{E}_{a,d}\cap\{0,...,N-1\}\right|.
$$
Note that since for every $a$ in $\N$, the set $\{\mathcal{E}_{a,d}\}_{d\in\Z}$ is a partition of $\N$, then $\mu_a$ is a probability measure on $\Z$.
\end{definition}

We call them correlation measures for the following reason. In \cite{besineau_independance_1973}, the author studied the statistical independence of sets defined by sum-of-digits functions. To that end, the classical tool in number theory is the correlation function of an arithmetic function. More precisely, let $\alpha$ be a real parameter and define the following map
$$
\begin{array}{llll}
f_{\alpha} :& \N &\longrightarrow &\C\\
            & n  &\longmapsto     &e^{2i\pi \alpha s_2(n)}
\end{array}
$$
and its autocorrelation function $\gamma_{f_{\alpha}}$ by
$$
\forall a \in \N,\quad \gamma_{f_{\alpha}}(a)=\lim_{N\rightarrow+\infty}\frac1N\sum_{n=0}^{N-1}f(n+a)\overline{f(n)}.
$$
Now notice that with our definitions, 
$$
\forall a \in \N,\quad \gamma_{f_{\alpha}}(a)=\sum_{d\in\Z}e^{2i\pi \alpha d}\mu_a(d).
$$

The study of sum-of-digits functions and of their statistical distribution has proven very fruitful in particular in number theory and probability. We can for instance quote the different classical works \cite{katai_distribution_1968} and  \cite{diaconis_distribution_1977} where it is respectively proven that the sum-of-digits functions are asymptotically normally distributed in any integer base and in binary via different methods. More recently, the authors of \cite{drmota_distribution_1998} and \cite{dumont_gaussian_1997} proved a local limit theorem independently for sum-of-digit functions in the general case. One of the most remarkable result on the subject of sum-of-digits functions in recent years is the solution of a problem of Gelfond on prime numbers by  Mauduit and Rivat in \cite{mauduit_sur_2010}. For a more thorough introduction on the subject, we advise the survey \cite{chen_distribution_2014}.

Our study differs in that the measures we study are different densities as those which are usually studied. Furthermore, it is usual to consider the sum-of-digits function as random variables by endowing sets of the form $\{0,...,n-1\}$ with the uniform probability measure, and then by increasing $n$ to get asymptotic properties. Here we do not want to be limited by the choice of the uniform probability measure on $\{0,...,n-1\}$ so we start by choosing a shift invariant ergodic probability measure on $\{0,1\}^\N$ and we look at the asymptotic properties of our measures by taking a sequence of prefixes of increasing length of a generic point for the ergodic measure. In the case where this measure is the symmetric Bernoulli measure, it is equivalent to looking at the asymptotic correlation measures where $s_2$ is seen as a random variable on $\{2^n,...,2^{n+1}-1\}$ endowed with the uniform probability measure. We proved that $\mu_a$ is generically normally distributed in this case in \cite{emme_central_2018}.

These new correlation measures hold some information that are yet to be fully understood in order to solve some number theory problems. For instance, we would like to mention a conjecture due to Cusick which he formulated shortly after working on a similar problem in \cite{cusick_combinatorial_2011}, though in a different formalism.

Let $a$ be a positive integer. Denote by $c_a$ the quantity
$$
c_a=\sum_{d\geq 0}\mu_a(d).
$$
It is conjectured that
$$\forall a \in \N, \quad c_a > \frac12$$and
$$\displaystyle\liminf_{a \in \N}=\frac12.$$

It was proved in \cite{morgenbesser_reverse_2012} that if the binary representation of $b$ is equal to the mirrored image of the binary representation of $a$, then $c_a=c_b$. It was also proved in \cite{drmota_conjecture_2016} that $\frac12$ is an accumulation point for $c_a$ by taking the sequence $\left (a_n\right )_{n\in\N}$ defined, for every $n$, by $a_n=\sum_{k=0}^n 4^k$. The fact that $\frac12$ is an accumulation point is also a consequence of the central limit theorem in \cite{emme_central_2018} and of the main result of this paper. Furthermore, our papers show, in particular, that $\frac12$ is an accumulation point for many sequences $(a_n)_{n\in\N}$ in a sense that will be made clear in the next subsection.

We would also like to point out that the probability measures $\mu_a$ hold some similarity with $2$-automatic sequences as defined in \cite{allouche_automatic_2003}. In particular, we recall the definition of Stern's sequence $(S_n)_{n\in\N}$ which is given by
$$
S_0=1;\qquad S_{2n}=S_{n}; \qquad S_{2n+1}=S_{n}+S_{n+1}
$$
and whose recurrence relations are quite close to the ones defining the correlation measures $\mu_a$ (as given at the beginning of Subsection~\ref{s.prop_mu}. We would like to quote the recent works in \cite{bettin_statistical_2017} where the authors proved a central limit theorem for the logarithms of Stern's sequence using transfer operators, but only in the case of the uniform law on the numbers in the set $\{0,...,n-1\}$.

\subsection{Result}

We are interested in the asymptotic behaviour of such correlation measures in the following sense:
\begin{itemize}
\item Let us consider the measured dynamical system $(\{0,1\}^\N,\mathcal{B},\nu,\sigma)$ where $\sigma$ is the left-shift on sequences in $\{0,1\}^\N$ and $\nu$ is a shift-invariant ergodic probability measure;

 \item For any sequence $X$ in $\{0,1\}^\N$, define $a_X(n)=\displaystyle\sum_{k=0}^n X_k \cdot 2^k$;
\end{itemize}
we wish to understand the behaviour of $\left(\mu_{a_X(n)} \right)_{n\in \N} $ for certain points $X$ which are generic for the measure $\nu$. We have shown in \cite{emme_central_2018} that if we take $\nu$ to be the symmetric Bernoulli measure, then  $\left(\mu_{a_X(n)} \right)_{n\in \N} $ satisfies a central limit theorem. In this paper we show the statement to be still true whenever $\nu$ is a shift-invariant ergodic measure. It should be noted that the method used here is very different from the one in \cite{emme_central_2018}, where the proof heavily relies on estimates of Bernoulli correlations. The central limit theorem being often obtained through some sort of independence, and having done only the case of Bernoulli measure by using this independence heavily, one could have expected that this result could only follow under assumptions of some sort on the measure $\nu$ like some mixing properties or asking for it to be a Gibbs measure. In this paper however, the only tool needed is Birkhoff's ergodic theorem for it to work in full generality.

\begin{theorem}\label{t.principal}
Let $\nu$ be a shift-invariant ergodic measure on $\{0,1\}^\N$ different from $\delta_{0^\infty}$ and $\delta_{1^\infty}$. For any $i$ in $\N$, let 
$$\mathcal{F}_i=\int \mathbbm{1}_{[0]}\times\mathbbm{1}_{[1]}\circ\sigma^i+
\mathbbm{1}_{[1]}\times\mathbbm{1}_{[0]}\circ\sigma^id\nu$$
and
$$
V_\nu=\sum_{i=1}^\infty \frac{\mathcal{F}_i}{2^i}.
$$
Then, defining the probability measure $\widetilde{\mu}_{X,n}$ on $\R$ by
$$
\forall x \in \R, \quad \widetilde{\mu}_{X,n}(x)=\mu_{a_X(n)}(\sqrt{V_\nu n}\cdot x)
$$
yields
$$
\widetilde{\mu}_{X,n}\underset{n\rightarrow +\infty}{\overset{weak}{\longrightarrow}}\varphi
 $$
 where
 $$
 \begin{array}{cccc}
  \varphi :& \R &\rightarrow &\R \\
           &  t & \mapsto & \frac{1}{\sqrt{2\pi}}e^{\frac{-1}{2}t^2}
 \end{array}
 $$

\end{theorem}

\begin{remark}
We would like to underline the fact even though, for simplicity and aesthetic reasons, this theorem is written in a way that involves ergodic measures, our proof works for a more general setting. Indeed, the theorem remains true for any element of the full shift $X$ such that for any $i$, the ergodic averages of the functions $\mathbbm{1}_{[0]}\times\mathbbm{1}_{[1]}\circ\sigma^i+
\mathbbm{1}_{[1]}\times\mathbbm{1}_{[0]}\circ\sigma^i$ converge along the orbit of $X$. This is a weaker assumption. Our theorem could then be viewed as a corollary of this statement, using Birkhoff's ergodic theorem.
\end{remark}

\subsection{Outline of the paper}

The goal of this paper is to prove Theorem~\ref{t.principal} using the moment's method, due to Markov in \cite{markov_demonstration_1913}, that is, computing the moments of $\widetilde\mu_{X,n}$ and showing that they converge towards the moments of the normal law $\mathcal{N}(0,1)$.

We start in Subsection~\ref{s.prop_mu} by recalling the main properties of the measures $\mu_a$ which are useful to our study. Namely the recurrence relations satisfied by the correlation measures $\mu_{a_X(n)}$ and the matricial representation of their characteristic functions given by these recurrence relations. For a more in depth study of those measures, we refer the reader to \cite{emme_asymptotic_2017}.

We recall shortly in Subsection~\ref{s.variance} how to compute the variance of $\mu_a$. This has to be done in order to know the constant of renormalisation for a subsequence of $\left(\mu_a\right)_{a \in \N}$ so as to get a central limit theorem. Computations are exposed here briefly and are done more thoroughly in \cite{emme_central_2018}.

In Section~\ref{s.rappel_contrib_moments} we recall how we compute the different moments: given the expression of the characteristic function as a product of matrices, we take the Taylor expansion of every matrix and compute the moments by classifying the terms involved as products of matrices. We study how the matrices give importances to some terms or not so as to know what will be killed by the renormalisation and which terms to focus on. In particular, this section gives all the tools to show that the moments of odd order converge towards 0, which is explained in Subsection~\ref{s.moments_impairs}.

Section~\ref{s.moments} is devoted to studying all the moments of the renormalised law. We wish to show that all the moments converge towards the moments of the Normal law $\mathcal{N}(0,1)$. We explain what are the quantities we wish to study in order to show that.


We underline the fact that Subsection~\ref{s.moments_2r}, is the most crucial and technical part of this paper. It is also the part of the paper where we truly use new techniques that do not appear in \cite{emme_central_2018}. In it, we show that the moments of even order of the renormalised law go towards the desired quantity; namely that the moments of order $2r$ go to $\frac{(2r)!}{2^r\cdot r!}$ at the limit.

\subsection*{Acknowledgements}
We would like to thank Sébastien Gouëzel for his interest in this problem and his decisive advice: being able to go through the computations of the moment of order 4 indeed gave us all the necessary arguments for the computation of the moments of any even order.

\section{The correlation measures $\mu_a$}\label{s.rappel_mu}

\subsection{Main properties}\label{s.prop_mu}

We start by recalling some of the most important properties of the correlation measures $\mu_a$. The most important one is the following:

\begin{proposition}\label{p.rec}
The correlation measures $\mu_a$ satisfy the following recurrence relations.
\begin{align}
\forall a \in \N,\quad \mu_{2a}&=\mu_a \label{e.rec_paire}\\
\forall (a,d)\in \N\times\Z,\quad \mu_{2a+1}(d)&=\frac12\left(\mu_a(d-1)+\mu_{a+1}(d+1) \right) \label{e.rec_impaire}
\end{align}

\end{proposition}

\begin{proof}
%
The proof of this proposition is done in \cite{emme_central_2018}. It follows from simple recurrence relations on the sets $\mathcal{E}_{a,d}$ defined in the introduction.

\end{proof}

Let $a$ be in  $\N$ with $\bin{a}=a_n...a_0$. The characteristic function of $\mu_a$, denoted $\widehat{\mu_a}$ is defined in the standard way:
$$
\forall \theta \in [0,2\pi),\ \widehat{\mu_a}(\theta)=\sum_{d \in \Z}e^{id\theta}\mu_a(d)
$$

\begin{proposition}\label{p.charac_matrix}
The characteristic function $\widehat{\mu}_a:[0,2\pi)\rightarrow \C$ is given by:
$$
\forall \theta \in [0,2\pi),\ \widehat{\mu}_a(\theta)=\begin{pmatrix}1 & 0 \end{pmatrix} \; \hA_{a_0}(\theta)\cdots \hA_{a_{n-1}}(\theta) \hA_{a_n}(\theta)\begin{pmatrix} \widehat{\mu_0}(\theta) \\ \widehat{\mu_1}(\theta) \end{pmatrix}
$$
where
$$
  \forall \theta \in [0,2\pi), \ 
  \hA_0(\theta):= \begin{pmatrix} 1 &  0  \\\frac12 e^{i \theta}& \frac12 e^{-i \theta} \end{pmatrix}, \ 
   \hA_1(\theta):= \begin{pmatrix} \frac12 e^{i \theta} & \frac12 e^{-i \theta}\\  0 & 1 \end{pmatrix}.
$$
\end{proposition}

\begin{proof}

Computing the Fourier transform of the measures $\mu_a$ and using Proposition \ref{p.rec} yields the result.
\end{proof}

\begin{proposition}
For any $a$ in $\N$, the measure $\mu_a$ has 0 mean.
\end{proposition}

\begin{proof}
We start by noticing that $\mu_0$ and $\mu_1$ both have 0 mean. This is obvious for $\mu_0$ since $\mu_0=\delta_0$. A simple calculation yields $\mu_1=\sum_{d\leq 1}\left(\frac{1}{2}\right)^{2-d}\delta_d$. We can the compute the mean of $\mu_1$ and find 0. Computing the means of $\mu_a$ with the recurrence relations of Proposition~\ref{p.rec} yields the result.
\end{proof}

For another proof of this statement, the reader can refer to \cite[Lemma 3.2.3]{emme_asymptotic_2017}.

\subsection{Computation of the generic variance}\label{s.variance}

In this section we use the expression of the characteristic function from Proposition~\ref{p.charac_matrix} in order to compute the variance of $\mu_a$. Recall that the variance of $\mu_a$ is the opposite of the second derivative of $\widehat{\mu}_a$ in 0. We wish to compute the Taylor expansion of the characteristic function at order 2 in order to get the variance.
A quick computation yields
$$
\forall \theta \in [0, 2\pi),\quad \widehat{\mu_0}(\theta)=1,\quad \widehat{\mu_1}(\theta)=\frac{e^{i\theta}}{2-e^{-i\theta}}
$$
and so,
$$
\forall \theta \in [0,2\pi),\ \widehat{\mu_a}(\theta)=\begin{pmatrix}1 & 0 \end{pmatrix} \; \hA_{a_0} \cdots \hA_{a_{n-1}} \hA_{a_n}\begin{pmatrix} 1 \\ \frac{e^{i\theta}}{2-e^{-i\theta}}\end{pmatrix}.
$$
Note that 
$$
\frac{e^{i\theta}}{2-e^{-i\theta}} =1-\theta^2+O(\theta^3).
$$
Let us now define the matrices playing a role in the Taylor expansion of $\widehat{\mu_a}$ near 0.
$$
  I_0 = \begin{pmatrix}1 &0 \\\frac12&\frac12\end{pmatrix}, \quad\; 
  \alpha_0 = \frac{1}2 \begin{pmatrix}0&0\\ 1&-1\end{pmatrix}, \quad\; 
  \beta_0 = \frac12 \begin{pmatrix}0&0\\ 1&1\end{pmatrix}, 
$$
$$
  I_1 = \begin{pmatrix}\frac12 &\frac12\\0&1\end{pmatrix}, \quad\; 
  \alpha_1 = \frac{1}2 \begin{pmatrix}1&-1\\0&0\end{pmatrix}, \quad\; 
  \beta_1 = \frac12 \begin{pmatrix}1&1\\0&0\end{pmatrix}.
$$
Indeed, we have:

$$
\hA_j(\theta)=I_j + i\theta \alpha_j-\frac12 \theta^2 \beta_j + O(\theta^3).
$$
with $j \in \{0,1\}$.

Recall that, since $\mu_a$ has 0 mean,
\[
 \widehat{\mu}_a(\theta)=\widehat{\mu}_a(0)-\frac{\text{Var}(\mu_a)}{2}\theta^2+o(\theta^2)
\]
hence, in order to get the variance of $\mu_a$, we want to compute the coefficient of order 2 in this product, up to a factor 2. We either get a quadratic term by multiplying two matrices of order 1 (matrices denoted by $\alpha$) or by choosing one matrix of order 2 (denoted by $\beta$). Notice that
$$
\alpha_0\begin{pmatrix}1 \\1 \end{pmatrix}= \alpha_1\begin{pmatrix}1 \\1 \end{pmatrix}=\begin{pmatrix}0 \\0 \end{pmatrix}
$$
and that
$$
I_0\begin{pmatrix}1 \\1 \end{pmatrix}=I_1\begin{pmatrix}1 \\1 \end{pmatrix}=\begin{pmatrix}1 \\1 \end{pmatrix}.
$$
Hence in order to compute the coefficient of order 2 in the Taylor expansion it is enough to consider the terms given by the matrices $\beta$ (and the quadratic term given by $\widehat{\mu}_1$)

 Let us give the change of basis that simultaneously trigonalise the matrices $I_0$ and $I_1$.
 
 Let us note
 $$
 P:=\begin{pmatrix}
     1  & 1 \\
     -1 & 1 
    \end{pmatrix}, \quad 
    P^{-1}=\frac12\begin{pmatrix}
		   1 &-1\\
		   1 & 1
		   \end{pmatrix}
 $$
 and compute
 $$\forall j \in \{0,1\}, \ 
 \widetilde{I_j}:=P I_j P^{-1}=\begin{pmatrix}
                  1 & \frac{(-1)^{j+1}}{2}\\
                  0 & \frac12
         
                 \end{pmatrix}, \quad
 \widetilde{\beta_j}:=P \beta_j P^{-1}=\begin{pmatrix}
                  \frac12 & 0\\
                  -\frac{(-1)^{j+1}}{2}& 0
                 \end{pmatrix}
 $$
 and
 $$
 P\begin{pmatrix}1 \\1 \end{pmatrix}=\begin{pmatrix}2 \\0 \end{pmatrix}, \quad P\begin{pmatrix}0 \\2 \end{pmatrix}=\begin{pmatrix}2 \\2 \end{pmatrix}, \quad 
 \begin{pmatrix}1 & 0 \end{pmatrix}P^{-1}=\begin{pmatrix}\frac12 & -\frac12 \end{pmatrix}.
 $$
 
  With this change of basis, the variance becomes:
 $$
 \text{Var}(\mu_a)=\begin{pmatrix}\frac12 & -\frac12 \end{pmatrix}\left(\widetilde\beta_{a_0}+\widetilde I_{a_0}\widetilde\beta_{a_1}+...+\widetilde I_{a_0}\cdots \widetilde I_{a_{n-1}}\widetilde\beta_{a_n}\right)\begin{pmatrix}2 \\ 0 \end{pmatrix} +\begin{pmatrix}\frac12 & -\frac12 \end{pmatrix}\widetilde I_{a_0}\cdots \widetilde I_{a_{n}}\begin{pmatrix}2 \\ 2 \end{pmatrix}.
 $$
 
Let us now introduce some notations that will be used throughout the paper. Let $a$ be a positive integer such that 
$$
a=\sum_{k=0}^{n-1}a_k 2^k
$$ 
is its standard binary expansion. Let us define the tuple $(b_0,...,b_{n-1})$ by
$$
\forall i \in \{0,...,n-1\},\quad b_i:=2a_i-1
$$
and notice that $b_i$ is in $\{-1,1\}$.
With these notations, we get 
 $$\forall i \in \{0,1\}, \ 
 \widetilde{I_{a_i}}:=P I_j P^{-1}=\begin{pmatrix}
                  1 & \frac{b_i}{2}\\
                  0 & \frac12
         
                 \end{pmatrix}, \quad
 \widetilde{\beta_{a_i}}:=P \beta_j P^{-1}=\begin{pmatrix}
                  \frac12 & 0\\
                  -\frac{b_i}{2}& 0
                 \end{pmatrix}.
 $$
 
 With this change of basis and these notations, it is easy to compute the variance. In the following proposition, we recall the result from \cite{emme_central_2018}, where the computations are made explicit.

\begin{proposition}\label{p.variance}
For any $a \in \N$ where $a=\sum_{k=0}^{n-1}a_k 2^k$, denote, for any $j\in \{0,...,n\},\ b_j=2a_j-1$. The variance of $\mu_a$ is given by the following:
  $$
 \text{Var}(\mu_a)=\frac{n}{2}+1-\frac{1}{2^{n}}-\frac12 \sum_{i=1}^{n-1}\sum_{k=0}^{n-1-i}\frac{b_{k+i}b_k}{2^{i}}+\sum_{k=0}^{n-1}\frac{b_{k}+b_{n-1-k}}{2^{k+1}}.
 $$
\end{proposition}

From this expression, we can compute the "asymptotic variance".
\begin{theorem}\label{t.asymptotic_variance}
Let $\nu$ be a shift-invariant ergodic probability measure on $\{0,1\}^\N$. 
$$
\nu-a.e. X,\quad \lim_{n\rightarrow+\infty}\frac{\text{Var}(\mu_{a_X(n)})}{n} =\sum_{i=1}^\infty \frac{\mathcal{F}_i}{2^i}.
$$
where
$$\mathcal{F}_i=\int \mathbbm{1}_{[0]}\times\mathbbm{1}_{[1]}\circ\sigma^i+
\mathbbm{1}_{[1]}\times\mathbbm{1}_{[0]}\circ\sigma^id\nu.$$
We denote by $V_\nu$ the quantity $ \sum_{i=1}^\infty \frac{\mathcal{F}_i}{2^i}$ and refer to it as the asymptotic variance.
\end{theorem} 

\begin{proof}
First, note that it is clear that, regardless of the choice of $X$ in $\{0,1\}^\N$, the quantity
$$
\frac1n \left( 1-\frac{1}{2^{n}}+\sum_{k=0}^{n-1}\frac{b_{k}+b_{n-1-k}}{2^{k+1}}  \right)
$$
goes to 0 as $n$ goes to $+\infty$. Thus we need to understand the limit, if it exists, of the sequence $\left ( \frac1n \left( \frac{n}{2}-\frac12 \sum_{i=1}^{n-1}\sum_{k=0}^{n-1-i}\frac{b_{k+i}b_k}{2^{i}}\right )\right )_{n\in\N}$.

First, let us denote, for all $i$ in $\N$,
$$
\begin{array}{rccl}
f_i &: \{0,1\}^\N & \longrightarrow & \N\\
    &  X       & \longmapsto     & \left(\mathbbm{1}_{[0]}\times\mathbbm{1}_{[1]}\circ\sigma^i+
    \mathbbm{1}_{[1]}\times\mathbbm{1}_{[0]}\circ\sigma^i\right)(X)
\end{array}.
$$
This allows us to write

$$
\frac{n}{2}- \frac12 \sum_{i=1}^{n-1}\sum_{k=0}^{n-1-i}\frac{b_{k+i}b_k}{2^{i}}=\frac{n}{2}- \frac12\left ( \sum_{i=1}^{n-1}\sum_{k=0}^{n-1-i}\frac{1}{2^{i}} - 2  \sum_{i=1}^{n-1}\sum_{k=0}^{n-1-i}\frac{f_i(\sigma^k(X))}{2^i}\right )
$$
since $b_kb_{k+i}$ can either be $1$ or $-1$ and since we defined $f_i$ so as to satisfy the relation $f_i(\sigma^k(X))=\delta_{-1}(b_kb_{k+i})$. Hence we have
$$
\frac{n}{2}- \frac12 \sum_{i=1}^{n-1}\sum_{k=0}^{n-1-i}\frac{b_{k+i}b_k}{2^{i}}
=
\frac{n}{2}- \frac12\left ( \sum_{i=1}^{n-1}\frac{n-i}{2^{i}} - 2  \sum_{i=1}^{n-1}\sum_{k=0}^{n-1-i}\frac{f_i(\sigma^k(X))}{2^i}\right )
$$
so
$$
\frac{n}{2}- \frac12 \sum_{i=1}^{n-1}\sum_{k=0}^{n-1-i}\frac{b_{k+i}b_k}{2^{i}}
=
\frac{n}{2}\left(1 - \sum_{i=1}^{n-1}\frac{1}{2^{i}}\right) +\sum_{i=1}^{n-1}\frac{i}{2^{i}} +  \sum_{i=1}^{n-1}\sum_{k=0}^{n-1-i}\frac{f_i(\sigma^k(X))}{2^i}.
$$
Now, notice that
$$
\lim_{n\rightarrow +\infty}\frac{1}{n}\left(\frac{n}{2}\left(1 - \sum_{i=1}^{n-1}\frac{1}{2^{i}}\right) +\sum_{i=1}^{n-1}\frac{i}{2^{i}} \right)=0
$$
hence we have
$$
\text{Var}(\mu_{a_X(n)})= \sum_{i=1}^{n-1}\sum_{k=0}^{n-1-i}\frac{f_i(\sigma^k(X))}{2^i}+o(n).
$$

Now, let $P$ be a positive integer, let $n$ be greater than $P$. Notice that

$$
\left|
\sum_{i=1}^{n-1}\sum_{k=0}^{n-1-i}\frac{f_i(\sigma^k(X))}{2^i}
 -
\sum_{i=1}^{P}\sum_{k=0}^{n-1-i}\frac{f_i(\sigma^k(X))}{2^i}
\right|
\leq
\sum_{i=P+1}^{n}\frac{n-i}{2^i}
$$
since $f_i$ takes values in $\{0,1\}$. Hence
$$
\left|
\sum_{i=1}^{n-1}\sum_{k=0}^{n-1-i}\frac{f_i(\sigma^k(X))}{2^i}
 -
\sum_{i=1}^{P}\sum_{k=0}^{n-1-i}\frac{f_i(\sigma^k(X))}{2^i}
\right|
\leq
\frac{n}{2^P}
$$
so, dividing by $n$ and taking the limit as $n$ goes to infinity yields
$$
\left|
\lim_{n\rightarrow +\infty} \frac{1}{n} \sum_{i=1}^{n-1}\sum_{k=0}^{n-1-i}\frac{f_i(\sigma^k(X))}{2^i} 
-
\sum_{i=1}^{P}\frac{\mathcal{F}_i}{2^i}
\right|
 \leq
\frac{1}{2^P}
$$
in virtue of Birkhoff's ergodic theorem. This being true for any $P$, the asymptotic variance for a given generic point $X$ is given by
$$
V_\nu=\sum_{i=1}^\infty \frac{\mathcal{F}_i}{2^i}.
$$
\end{proof}
 
\begin{remark}
At this point, we would like to make the remark that taking $\nu$ different from $\delta_{0^\infty}$  and $\delta_{1^\infty}$ ensures that $V_\nu$ is positive.
\end{remark}

\section{Determining which terms contribute to the moments}\label{s.rappel_contrib_moments}

The goal of this section is to understand which products of matrices involved in the Taylor expansions of each $\hat A_0$ and $\hat A_1$ contribute to the moments of the renormalised measure.


A detailed analysis of the moments and especially of the contribution of each terms in the matrices is done in \cite[Section 3.5]{emme_central_2018}. We briefly recall however the main ideas of this study.

First, remark that the Taylor expansions $\hat A_i$ can be written as a linear combination of $I_i$, $\alpha_i$ and $\beta_i$, $I_i$ being the constant term, $\alpha_i$ being a matrix coefficient of terms of odd orders and $\beta_i$ being a matrix coefficient of terms of even orders.

In order to compute a moment of given order $r$ at time $n$, one has, in particular, to develop the product $\hat A_{X_0}(\theta) \cdots \hat A_{X_{n-1}}(\theta)$ and to understand the matrix that is multiplied by $\theta^r$. This matrix is expressed as a sum of products of matrices which we can sort into types depending on how many matrices $\alpha_i$ and $\beta_i$ are involved the product. For instance, we say that the product $I_{X_0}\beta_{X_1}I_{X_2}I_{X_3}\beta_{X_4}I_{X_5}\alpha_{X_6}$ is of type $\left ( \alpha^1,\beta^2\right )$ because there is one matrix of type $\alpha$ and two matrices of type $\beta$ in the product. 

Let us introduce some formalism to make this precise. Given an element of the full-shift $X$ let us define the map $M_X$ from the free monoid $\{I,\alpha,\beta\}^*$ to $\mathcal{M}_2(\R)$ such that
$$
M_X(I)=I_{X_0},\quad M_X(\alpha)=\alpha_{X_0}, \quad M_X(\beta)=\beta_{X_0}
$$
and, for any two elements $u$ and $v$ in $\{I,\alpha,\beta\}^*$, $u$ being of length $n$,
$$
M_X(uv)=M_X(u)M_{\sigma^{n}(X)}(v).
$$
We denote by $|u|$ the length of $u$, and by $|u|_l$ the number of occurrences of the symbol $l$ in $u$.

We now introduce the set 
$$
\mathcal{F}^{(n)}_{\left (\alpha^p,\beta^q \right )}(X)=\{M_X(u)\, \text{ such that } \, |u|=n, \, |u|_\alpha = p, \, |u|_\beta =q \}
$$
which is just the set of matrices of type $\left (\alpha^p,\beta^q \right )$ which appear in the Taylor expansion of the product $\hat A_{X_0}(\theta) \cdots \hat A_{X_{n-1}}(\theta)$.

Since we mostly work with a fixed $X$ -- like in this section -- we will often write $\mathcal{F}^{(n)}_{\left (\alpha^p,\beta^q \right )}$ from now on, except when several elements of the full shift should be considered, which will be made explicit.

\begin{definition}
 We say that a type $\left (\alpha^p,\beta^q \right )$ contributes with weight at most $k$ if
$$
\sum_{M\in \mathcal{F}^{(n)}_{\left (\alpha^p,\beta^q \right )}} \|M\| = O(n^k)
$$
where $\|\cdot\|$ is the maximal $l^1$ norm on rows, which defines a submultiplicative norm.
\end{definition}

\begin{lemma}\label{l.weight}
For any pair of non-negative integers $p$ and $q$, the type $\left (\alpha^p,\beta^q \right )$ contributes with weight at most $q$.
\end{lemma}

\begin{proof}
This lemma is proved in details in \cite[Lemma 3.7]{emme_central_2018}. We recall very briefly the main ingredients of this proof.

First, notice that the type $\left (\alpha^0,\beta^q \right )$ contributes with weight at most $q$ since $\|I_i\|=\|\alpha_i\|=\|\beta_i\|=1$ and since $\|\cdot\|$ is submultiplicative, hence $
\displaystyle\sum_{M\in \mathcal{F}^{(n)}_{\left (\alpha^0,\beta^q \right )} } \|M\| \leq \dbinom{n}{q}
$. 

Then, the proof that the type $\left (\alpha^p,\beta^q \right )$ also contributes with weight at most $q$ is done by induction on $p$ and relies solely on the facts that $\alpha_i I_j=\frac12\alpha_i$ and that $\|\cdot\|$ is submultiplicative.
\end{proof}

\section{Computing the moments}\label{s.moments}

Let us fix once and for all $\nu$ to be an ergodic measure different from $\delta_{0^\infty}$  and $\delta_{1^\infty}$ as well as a generic point $X$ in $\{0,1\}^\N$.

Recall that we define the measure $\widetilde{\mu}_{X,n}$ on $\R$ as
$$
\forall x \in \R, \, \widetilde{\mu}_{X,n}(x)=\mu_{a_X(n)}(\sqrt{V_\nu\cdot n}x)
$$
hence
$$
\forall \theta \in \R,  \widehat{\widetilde{\mu}}_{X,n}(\theta)=\widehat\mu_{a_X(n)}\left(\frac{\theta}{\sqrt{V_\nu\cdot n}}\right).
$$
Denoting by $(m_r(n))$ the moment of order $r$ of $\mu_{a_X(n)}$, we remind that, around $0$
$$
\widehat\mu_{a_X(n)}\left(\theta\right)=\sum_{r=0}^{+\infty}\frac{i^rm_r(n)}{r!}\theta^{r}
$$
hence
$$
 \widehat{\widetilde{\mu}}_{X,n}(\theta)=\sum_{r=0}^{+\infty}\frac{i^rm_r(n)}{(V_\nu n)^{\frac{r}{2}}r!}\theta^{r}
$$
thus we need to prove that
$$
\lim_{n\rightarrow +\infty}\frac{m_{2r}(n)}{(V_\nu n)^r}=\frac{(2r)!}{2^rr!}
$$
and
$$
\lim_{n\rightarrow +\infty}\frac{m_{2r+1}(n)}{(V_\nu n)^{\frac{2r+1}{2}}}=0.
$$

\subsection{Moments of odd order}\label{s.moments_impairs}

In this section we prove that 
$$
\lim_{n\rightarrow +\infty}\frac{m_{2r+1}(n)}{(V_\nu n)^{\frac{2r+1}{2}}}=0.
$$
Notice that in order to get a term for the moment of odd order $2r+1$, one must necessarily have some matrix of type $\alpha$ in the Taylor expansion. Then remark that the products of type $(\alpha^1,\beta^r)$ is the one which contributes with highest weight, that is $r$, from Lemma~\ref{l.weight}.

But then, since in the renormalisation, there appears a division by $\sqrt{n}^{2r+1}$, and since there is a finite number of type of products, necessarily we have
$$
\lim_{n\rightarrow +\infty}\frac{m_{2r+1}(n)}{(V_\nu n)^{\frac{2r+1}{2}}}=0.
$$

\subsection{Moments of even order}\label{s.moments_2r}

We would like to insist on the fact that everything up to now was already present in \cite{emme_central_2018}, albeit written in a slightly different manner in some parts. This section, however, is the crucial part of the proof of Theorem~\ref{t.principal}. It is completely new and involves a fair amount of quite technical computations.

First of all, let us state that, in the light of Lemma~\ref{l.weight}, in order to compute the moments of even order $2r$, we only need to consider the terms of type $(\alpha^0, \beta^r)$ since they are the only ones of weight not less than $r$ and hence they are the only ones that could not be killed by the renormalisation.

To be more precise, since the Taylor expansion of $\widehat{\mu}_{a_X(n)}\left(\frac{\theta}{\sqrt{V_\nu n}} \right)$ near $0$ is given by
$$
\begin{pmatrix}1 & 0 \end{pmatrix}
\left( I_{X_0}+i\frac{\theta}{\sqrt{V_\nu n}}\alpha_{X_0}-\frac{\theta^2}{2V_\nu n}\beta_{X_0}+o(\theta^2) \right)
\cdots
\left( I_{X_{n-1}}+i\frac{\theta}{\sqrt{V_\nu n}}\alpha_{X_{n-1}}-\frac{\theta^2}{2V_\nu n}\beta_{X_{n-1}}+o(\theta^2) \right)
\begin{pmatrix}1 \\1 \end{pmatrix}
$$
and in view of the types contributions, we have
$$
\frac{(-1)^r m_{2r}(n)}{(V_\nu n)^r(2r)!}=\left(\frac{-1}{2V_\nu n}\right)^r \begin{pmatrix}1 & 0 \end{pmatrix}\sum_{M\in\mathcal{F}^{(n)}_{(\alpha^0,\beta^r)}}M\begin{pmatrix}1 \\1 \end{pmatrix}+o(1).
$$
Hence, in order to get
$$
\lim_{n\rightarrow +\infty}\frac{m_{2r}(n)}{(V_\nu n)^r}=\frac{(2r)!}{2^rr!}
$$
we must prove that
$$
\lim_{n\rightarrow+\infty}\frac{1}{n^r}\begin{pmatrix}1 & 0 \end{pmatrix}\sum_{M\in\mathcal{F}^{(n)}_{(\alpha^0,\beta^r)}}M\begin{pmatrix}1 \\1 \end{pmatrix}=\frac{V_\nu ^r}{r!}
$$
or, equivalently, with the basis change in the proof of Proposition \ref{p.variance}
$$
\lim_{n\rightarrow+\infty}\frac{1}{n^r}\begin{pmatrix}\frac12 & -\frac12 \end{pmatrix}\sum_{M\in\mathcal{F}^{(n)}_{(\alpha^0,\beta^r)}}PMP^{-1}\begin{pmatrix}2 \\0 \end{pmatrix}=\frac{V_\nu ^r}{r!}.
$$
We prove this using two main lemmas.

\begin{lemma}\label{l.matrix_form}
For any positive integers $n$ and $k$, the matrix $\displaystyle\sum_{M\in\mathcal{F}^{(n)}_{(\alpha^0,\beta^k)}}PMP^{-1}$
is of the form
$$
\sum_{M\in\mathcal{F}^{(n)}_{(\alpha^0,\beta^r)}}PMP^{-1}=\begin{pmatrix}
A_{n}(X,2r) & 0\\
B_{n}(X,2r) & 0
\end{pmatrix}
$$
and the coefficients satisfy
$$
\exists C_1^{(r)}>0,\quad \forall r  \in \N,\quad \forall n  \in \N,\quad |A_{n}(X,2r)|< C_1^{(r)}n^r
$$
and
$$
\exists C_2^{(r)}>0,\quad \forall r  \in \N,\quad \forall n  \in \N,\quad |B_{n}(X,2r)|<C_2^{(r)}n^{r-1}.
$$ 
where the constants $ C_1^{(r)}$ and $C_2^{(r)}$ do not depend on the choice of $X$.
\end{lemma}

\begin{proof}[Proof of Lemma \ref{l.matrix_form}]
We prove this lemma by induction on $k$.

The case $k=1$ can be obtained easily from the computations of Proposition \ref{p.variance} but, for simplicity, we expose it here as well.
First let us write explicitly
$$
\sum_{M\in\mathcal{F}^{(n)}_{(\alpha^0,\beta^1)}}PMP^{-1}
=
P\left(\beta_{X_0}+I_{X_0}\beta_{X_1}+...+I_{X_0}...I_{X_{n-2}}\beta_{X_{n-1}} \right)P^{-1}
$$
hence
$$
\sum_{M\in\mathcal{F}^{(n)}_{(\alpha^0,\beta^1)}}PMP^{-1}
=
\left(\widetilde\beta_{X_0}+\widetilde I_{X_0}\widetilde\beta_{X_1}+...+\widetilde I_{X_0}...\widetilde I_{X_{n-2}}\widetilde\beta_{X_{n-1}} \right).
$$
Now notice that, for any $j$ in $\llbracket 0,n-1 \rrbracket$,
 $$
 \widetilde I_{X_0}\cdots \widetilde I_{X_{j-1}}\widetilde\beta_{X_{j}}=\begin{pmatrix}
                                                \frac12-\frac{b_{j}}{2}\sum_{i=0}^{j-1}\frac{b_i}{2^{j-i}}&0\\
                                                -\frac{b_{j}}{2^{j+1}} & 0
                                               \end{pmatrix}
 $$
 Hence
 $$
 A_n(2)= \frac{n}{2}-\sum_{j=0}^{n-1}\frac{b_{j}}{2}\sum_{i=0}^{j-1}\frac{b_i}{2^{j-i}}
 $$
 which indeed satisfies, independently from $X$, that
 $$
 |A_n(2)|<2n.
 $$
 Moreover,
 $$
 B_n(2)=-\sum_{j=0}^{n-1}\frac{b_{j}}{2^{j+1}}
 $$
 which converges, uniformly in $X$, as $n$ goes to infinity so we do have that
 $$
 |B_n(2)|<1.
 $$
 
 Let us now assume that this property holds for a given fixed integer $r$.

 Notice that 
 $$
 \sum_{M\in\mathcal{F}^{(n)}_{(\alpha^0,\beta^{r+1})}}PMP^{-1}
 =
 \sum_{m=0}^{n-1}\widetilde I_{X_0}...\widetilde I_{X_{m-1}}\widetilde \beta_{X_m}\left(\sum_{M\in\mathcal{F}^{(n-m)}_{(\alpha^0,\beta^r)}(\sigma^m(X))}PMP^{-1}\right).
 $$
Now, since the coefficients of the matrices $\left(I_{X_0}...I_{X_{m-1}}\right)_{m\in\N}$ are uniformly bounded in $m$, and since, by the induction hypothesis we have
$$
\sum_{M\in\mathcal{F}^{(n-m)}_{(\alpha^0,\beta^r)}(\sigma^m(X))}PMP^{-1}
=
\begin{pmatrix}
A_{n-m}(\sigma^m(X),2r)&0\\
B_{n-m}(\sigma^m(X),2r)&0
\end{pmatrix},
$$
writing
$$
 \sum_{M\in\mathcal{F}^{(n)}_{(\alpha^0,\beta^{r+1})}}PMP^{-1}
 =
 \sum_{m=0}^{n-1}
\begin{pmatrix}
\frac12-\frac{b_{m}}{2}\sum_{i=0}^{m-1}\frac{b_i}{2^{m-i}}&0\\
-\frac{b_{m}}{2^{m+1}} & 0
\end{pmatrix}
\begin{pmatrix}
A_{n-m}(\sigma^m(X),2r)&0\\
B_{n-m}(\sigma^m(X),2r)&0
\end{pmatrix}
$$
yields the desired result.

\end{proof}

\begin{remark}
It should be stressed that since we are interested in computing the following quantity, if it exists
$$
\lim_{n\rightarrow+\infty}\frac{1}{n^r}\begin{pmatrix}\frac12 & -\frac12 \end{pmatrix}\sum_{M\in\mathcal{F}^{(n)}_{(\alpha^0,\beta^r)}}PMP^{-1}\begin{pmatrix}2 \\0 \end{pmatrix},
$$
the point of Lemma \ref{l.matrix_form} is to show that in order to understand the moment of order $2r$, we must only understand the limit of the sequence $\left(\frac{1}{n^r}A_{n}(X,2r) \right)_{n\in\N}$ since $\displaystyle\lim_{n\rightarrow +\infty} \frac{1}{n^r}B_{n}(X,2r)=0$ from the lemma.
\end{remark}

For simplicity in the upcoming computations, let us introduce the following notation.\\
\medskip\\
$
\forall X \in \{0,1\}^\N,\ \forall r\geq 1,\ \forall n > r,
$
\flushright{$\displaystyle S_n(X,p_1,...,p_r):=\sum_{0\leq j_1<...<j_r\leq n-1} f_{p_1}(\sigma^{p_2+...+p_r+j_r}(X))
f_{p_2}(\sigma^{p_3+...+p_r+j_{r-1}}(X))...
f_{p_r}(\sigma^{j_1}(X))$.}
\flushleft
We now give the last crucial lemma for the demonstration of Theorem \ref{t.principal}.

\begin{lemma}\label{l.main}
There exists a constant $C>0$ and for every integer $P$, there exists a constant $K_P$ such that, for every $n$ in $\N$, 
$$
\left| A_{n}(X,2r) - \sum_{1\leq p_1...p_r \leq P} \frac{1}{2^{p_1+...+p_r}}
S_n(X,p_1,...,p_r)\right| 
\leq
K_P n^{r-1}+C\frac{n^r}{2^P}
$$
and $C$ and $K_P$ do not depend on the choice of $X$.
\end{lemma}

\begin{remark}
Before proving this lemma, we would like to explain why it is useful for the proof of Theorem~\ref{t.principal}. Here, $P$ is a fixed parameter, and for any such fixed parameter, one can bound the distance between $\displaystyle\lim_{n\rightarrow +\infty} \frac1{n^r}A_n(X,2r)$  -- which is the quantity we wish to understand -- and $\lim_{n\rightarrow +\infty} \frac{1}{n^r}\sum_{1\leq p_1...p_r \leq P} \frac{1}{2^{p_1+...+p_r}}
S_n(X,p_1,...,p_r)$ which we know how to compute. An argument similar to the one in the proof of Theorem~\ref{t.asymptotic_variance} then gives us the desired result.
\end{remark}

\begin{proof}[Proof of Lemma~\ref{l.main}]
We prove this by induction on $r$. The case $r=1$ can easily be obtained from the computations in the proof of Theorem \ref{t.asymptotic_variance}.

We assume this is true for a given integer $r$.
Since 
 $$
 \sum_{M\in\mathcal{F}^{(n)}_{(\alpha^0,\beta^{r+1})}}PMP^{-1}
 =
 \sum_{m=0}^{n-1}\widetilde I_{X_0}...\widetilde I_{X_{m-1}}\widetilde \beta_{X_m}\left(\sum_{M\in\mathcal{F}^{(n-m)}_{(\alpha^0,\beta^r)}(\sigma^m(X))}PMP^{-1}\right).
 $$
 and
  $$
 \widetilde I_{X_0}\cdots \widetilde I_{X_{m-1}}\widetilde\beta_{X_{m}}=\begin{pmatrix}
                                                \frac12-\frac{b_{m}}{2}\sum_{i=0}^{m-1}\frac{b_i}{2^{m-i}}&0\\
                                                -\frac{b_{m}}{2^{m+1}} & 0
                                               \end{pmatrix},
 $$
 we get that
 \begin{equation}\label{e.reccurence}
  A_{n}(X,2r+2)=\sum_{m=0}^{n-1}\left(\frac12-\frac{b_{m}}{2}\sum_{i=0}^{m-1}\frac{b_i}{2^{m-i}}\right)A_{n-m}(\sigma^m(X),2r)
 \end{equation}

From the induction relation, we know that
$$
\left| A_{n-m}(\sigma^m(X),2r) - \sum_{1\leq p_1...p_r \leq P} \frac{1}{2^{p_1+...+p_r}}
S_{n-m}(\sigma^m(X),p_1,...,p_r)\right| 
\leq
K_P n^{r-1}+C\frac{n^r}{2^P}
$$
independently from $m$.

Moreover, for every $X$, 
$$
\sum_{m=0}^{n-1}\left(\frac12-\frac{b_{m}}{2}\sum_{i=0}^{m-1}\frac{b_i}{2^{m-i}}\right)\leq  \frac32 n
$$ 
so 
\begin{multline}\label{e.bound_from_induction}
\left|A_{n}(X,2r+2)-
\sum_{m=0}^{n-1}\left(\frac12-\frac{b_{m}}{2}\sum_{i=0}^{m-1}\frac{b_i}{2^{m-i}}\right)
\sum_{1\leq p_1...p_r \leq P} \frac{1}{2^{p_1+...+p_r}}
S_{n-m}(\sigma^m(X),p_1,...,p_r)\right|\\
\leq
\frac32 \left(K_P n^{r}+2C\frac{n^{r+1}}{2^P}\right).
\end{multline}

Hence, in order to prove the induction step, it is enough to show that there exists $\widetilde C$ independent from $X$, $n$ and $P$, and $\widetilde K_P$, independent from $X$ and $n$ such that
\begin{multline}
\left|
\sum_{1\leq p_1...p_{r+1} \leq P} \frac{1}{2^{p_1+...+p_{r+1}}}
S_n(X,p_1,...,p_{r+1})\right.\\ \left.
-
\sum_{m=0}^{n-1}\left(\frac12-\frac{b_{m}}{2}\sum_{i=0}^{m-1}\frac{b_i}{2^{m-i}}\right)
\sum_{1\leq p_1...p_r \leq P} \frac{1}{2^{p_1+...+p_r}}
S_{n-m}(\sigma^m(X),p_1,...,p_r)
\right|
\leq
\widetilde K_P n^{r} + \widetilde C \frac{n^{r+1}}{2^P}.
\end{multline}
First, notice that the functions $(f_p)_{p\in \N}$ take values in $\{0,1\}$,
$$
\forall X \in \{0,1\}^\N, \ \forall n \in \N, \forall (p_1,...,p_r) \in \N^r, \quad S_n(X,p_1,...,p_r) < n^r.
$$
and that $\displaystyle\sum_{1\leq p_1,...,p_r\leq P}\frac{1}{2^{p_1+...+p_r}}=\left(1-\frac{1}{2^P}\right)^r$
hence, 
\begin{equation}\label{e.boun_last_terms}
\left|
\sum_{m=n-P}^{n-1}\left(\frac12-\frac{b_{m}}{2}\sum_{i=0}^{m-1}\frac{b_i}{2^{m-i}}\right)
\sum_{1\leq p_1...p_r \leq P} \frac{1}{2^{p_1+...+p_r}}
S_{n-m}(\sigma^m(X),p_1,...,p_r)
\right|
<
\frac{3}{2}Pn^r
\end{equation}
so, we only need to study the following quantity:
$$
\sum_{m=0}^{n-P-1}\left(\frac12-\frac{b_{m}}{2}\sum_{i=0}^{m-1}\frac{b_i}{2^{m-i}}\right)
\sum_{1\leq p_1...p_r \leq P} \frac{1}{2^{p_1+...+p_r}}
S_{n-m}(\sigma^m(X),p_1,...,p_r).
$$

To that end, we do the following computations
\begin{align*}
&\sum_{m=0}^{n-P-1}\left(\frac12-\frac{b_{m}}{2}\sum_{i=0}^{m-1}\frac{b_i}{2^{m-i}}\right)
\sum_{1\leq p_1...p_r \leq P} \frac{1}{2^{p_1+...+p_r}}S_{n-m}(\sigma^m(X),p_1,...,p_r)\\
=&\sum_{1\leq p_1...p_r \leq P} \frac{1}{2^{p_1+...+p_r}}
\left
(\sum_{m=0}^{n-P-1}\left(\frac12-\frac{b_{m}}{2}\sum_{i=0}^{m-1}\frac{b_i}{2^{m-i}}\right)
\right)
S_{n-m}(\sigma^m(X),p_1,...,p_r)\\
=&\sum_{1\leq p_1...p_r \leq P} \frac{1}{2^{p_1+...+p_r}}
\left
(\sum_{m=0}^{n-P-1}\frac{S_{n-m}(\sigma^m(X),p_1,...,p_r)}{2}-\sum_{m=0}^{n-P-1}S_{n-m}(\sigma^m(X),p_1,...,p_r)\frac{b_{m}}{2}\sum_{i=0}^{m-1}\frac{b_i}{2^{m-i}}\right).
\end{align*}
Now, notice that
\begin{align*}
&\sum_{m=0}^{n-P-1}S_{n-m}(\sigma^m(X),p_1,...,p_r)\frac{b_{m}}{2}\sum_{i=0}^{m-1}\frac{b_i}{2^{m-i}}\\
=&\sum_{m=0}^{n-P-1}\frac{S_{n-m}(\sigma^m(X),p_1,...,p_r)}{2}\left(1 - \frac{1}{2^{m-1}} \right) - \sum_{m=0}^{n-P-1}S_{n-m}(\sigma^m(X),p_1,...,p_r)\sum_{i=0}^{m-1}\frac{f_{m-i}(\sigma^i(X))}{2^{m-i}}\\
=&\sum_{m=0}^{n-P-1}\frac{S_{n-m}(\sigma^m(X),p_1,...,p_r)}{2}\left(1 - \frac{1}{2^{m-1}} \right) - \sum_{q=1}^{n-P-1}\frac{1}{2^q}\sum_{i=0}^{n-P-1-q}f_q(\sigma^i(X))S_{n-q-i}(\sigma^{q+i}(X),p_1,...,p_r)
\end{align*}
with the change of variables $q=m-i$.
Hence
\begin{align*}
&\sum_{1\leq p_1...p_r \leq P} \frac{1}{2^{p_1+...+p_r}}
\left(
\sum_{m=0}^{n-P-1}\frac{S_{n-m}(\sigma^m(X),p_1,...,p_r)}{2}-\sum_{m=0}^{n-P-1}S_{n-m}(\sigma^m(X),p_1,...,p_r)\frac{b_{m}}{2}\sum_{i=0}^{m-1}\frac{b_i}{2^{m-i}}\right)\\
=&\sum_{1\leq p_1...p_r \leq P} \frac{1}{2^{p_1+...+p_r}}
\left(
\sum_{m=0}^{n-P-1}\frac{S_{n-m}(\sigma^m(X),p_1,...,p_r)}{2^m}
+ \sum_{q=1}^{n-P-1}\frac{1}{2^q}\sum_{i=0}^{n-P-1-q}f_q(\sigma^i(X))S_{n-q-i}(\sigma^{q+i}(X),p_1,...,p_r)\right)
\end{align*}
and notice that 

\begin{equation}\label{e.bound_first_remainder}
\left|
\sum_{1\leq p_1...p_r \leq P} \frac{1}{2^{p_1+...+p_r}}
\sum_{m=0}^{n-P-1}\frac{S_{n-m}(\sigma^m(X),p_1,...,p_r)}{2^m}
\right|
\leq
2n^r.
\end{equation}

So far, from Equations (\ref{e.boun_last_terms}) and (\ref{e.bound_first_remainder}) ,we have computed that
\begin{multline}\label{e.bound_summary}
\left|\sum_{m=0}^{n-1}\left(\frac12-\frac{b_{m}}{2}\sum_{i=0}^{m-1}\frac{b_i}{2^{m-i}}\right)
\sum_{1\leq p_1...p_r \leq P} \frac{1}{2^{p_1+...+p_r}}
S_{n-m}(\sigma^m(X),p_1,...,p_r)\right.\\
\left.-\sum_{1\leq p_1...p_r \leq P} \frac{1}{2^{p_1+...+p_r}}
 \sum_{q=1}^{n-P-1}\frac{1}{2^q}\sum_{i=0}^{n-P-1-q}f_q(\sigma^i(X))S_{n-q-i}(\sigma^{q+i}(X),p_1,...,p_r)\right|
 \leq
 \frac{3P+4}2 n^r.
\end{multline}
Let us now bound
\begin{multline*}
\left|\sum_{1\leq p_1...p_r \leq P} \frac{1}{2^{p_1+...+p_r}}
 \sum_{q=1}^{n-P-1}\frac{1}{2^q}\sum_{i=0}^{n-P-1-q}f_q(\sigma^i(X))S_{n-q-i}(\sigma^{q+i}(X),p_1,...,p_r)\right.\\
 -\left.\sum_{1\leq p_1...p_{r+1} \leq P} \frac{1}{2^{p_1+...+p_{r+1}}}
S_n(X,p_1,...,p_{r+1})
 \right|.
\end{multline*}
First of all, notice that
$$
\left|\sum_{q=P+1}^{n-P-1}\frac{1}{2^q}\sum_{i=0}^{n-P-1-q}f_q(\sigma^i(X))S_{n-q-i}(\sigma^{q+i}(X),p_1,...,p_r)\right|\leq\frac{n^{r+1}}{2^P}
$$
hence
\begin{equation}\label{e.bound_last_terms_bis}
\left|\sum_{1\leq p_1...p_r \leq P} \frac{1}{2^{p_1+...+p_r}}\sum_{q=P+1}^{n-P-1}\frac{1}{2^q}\sum_{i=0}^{n-P-1-q}f_q(\sigma^i(X))S_{n-q-i}(\sigma^{q+i}(X),p_1,...,p_r)\right|\leq\frac{n^{r+1}}{2^P}.
\end{equation}

Finally, let us write
$$
S_n(X,q,p_1,...p_r)=\sum_{i=0}^{n-1}f_q(\sigma^i(X))S_{n-q-i}(\sigma^{q+i}(X),p_1,...,p_r)
$$
and notice  that 
\begin{equation}\label{e.bound_second_remainder}
\forall q \in \{1,...,P\}, \ \left|\sum_{i=n-P-q}^{n-1}f_q(\sigma^i(X))S_{n-q-i}(\sigma^{q+i}(X),p_1,...,p_r)\right|\leq 2Pn^r
\end{equation}

So, Equations (\ref{e.bound_last_terms_bis}) and (\ref{e.bound_second_remainder}) together yield
\begin{multline}\label{e.final_bound}
\left|\sum_{1\leq p_1...p_r \leq P} \frac{1}{2^{p_1+...+p_r}}
 \sum_{q=1}^{n-P-1}\frac{1}{2^q}\sum_{i=0}^{n-P-1-q}f_q(\sigma^i(X))S_{n-q-i}(\sigma^{q+i}(X),p_1,...,p_r)\right.\\
 \left.-\sum_{1\leq p_1...p_{r+1} \leq P} \frac{1}{2^{p_1+...+p_{r+1}}}
S_n(X,p_1,...,p_{r+1})
 \right|
 \leq 
 2Pn^r+\frac{n^{r+1}}{2^P}
\end{multline}
Finally, putting together Equations (\ref{e.bound_from_induction}), (\ref{e.bound_summary}) and (\ref{e.final_bound}) yields
$$
\left| A_{n}(X,2r+2) - \sum_{1\leq p_1...p_{r+1} \leq P} \frac{1}{2^{p_1+...+p_{r+1}}}
S_n(X,p_1,...,p_{r+1})\right| 
\leq
\left(\frac{3K_P+5P+4}{2}\right)n^r +\left(2C+1\right)\frac{n^{r+1}}{2^P}
$$
which proves the induction step and thus, the lemma.
\end{proof}

Let us finally state a classical lemma in ergodic theory.

\begin{lemma}\label{l.ergodic_k}
Let $(X,\mathcal{B},\nu,T)$ be a measured dynamical system, $\nu$ being a $T$ invariant ergodic probability measure. Let $g_1,...,g_r$ be functions in $L^1(X,\nu)$, then
$$
\nu-a.e. x \in X,\ \lim_{n\rightarrow +\infty} \frac{1}{n^r}\sum_{0\leq i_1<...<i_r \leq n-1}g_1\circ T^{i_1}(x)...g_r\circ T^{i_r}(x)=\frac{1}{r!}\prod_{j=1}^r\int_X g_j d\nu.
$$
\end{lemma}
\begin{proof}
This lemma is a consequence of Birkhoff's ergodic theorem. It can be proved by induction on $r$, using an Abel transform.
\end{proof}
 
 Now, as stated in the beginning of the present section, we must show that
 $$
\nu-a.e. X \in \{0,1\}^\N, \ \lim_{n\rightarrow+\infty}\frac{1}{n^r}\begin{pmatrix}\frac12 & -\frac12 \end{pmatrix}\sum_{M\in\mathcal{F}^{(n)}_{(\alpha^0,\beta^r)}(X)}PMP^{-1}\begin{pmatrix}2 \\0 \end{pmatrix}=\frac{V_\nu ^r}{r!}.
$$
In view of Lemma \ref{l.matrix_form}, we must actually show that
$$
\nu-a.e. X \in \{0,1\}^\N, \ \lim_{n\rightarrow+\infty}\frac{1}{n^r}A_n(X,2r)=\frac{V_\nu ^r}{r!}.
$$
Notice that, for any $P$ in $\N$, and any $(p_1,...,p_r)$ in $\{1,...,P\}^r$, Lemma \ref{l.ergodic_k} yields
$$
\nu-a.e. X \in \{0,1\}^\N, \ \lim_{n\rightarrow+\infty}\frac{1}{n^r}S_n(X,p_1,...,p_r)=\frac{1}{r!}\mathcal{F}_{p_1}...\mathcal{F}_{p_r}
$$
where we recall that $\mathcal{F}_i=\int f_i d\nu$.
There being a finite number of such tuples $(p_1,...,p_r)$, we have
$$
\nu-a.e. X \in \{0,1\}^\N, \ \lim_{n\rightarrow+\infty}\frac{1}{n^r}\sum_{1\leq p_1,..., p_r \leq P}\frac{1}{2^{p_1+...+p_r}}S_n(X,p_1,...,p_r)=\frac{1}{r!}\left(\sum_{p=1}^P\frac{\mathcal{F}_{p}}{2^p}\right)^r
$$
and we denote this limit $l_P$.
From Lemma \ref{l.main}, and in the same manner as for the computation of the asymptotic variance, we have
$$
\nu-a.e. X \in \{0,1\}^\N, \ \left|\lim_{n\rightarrow+\infty}\frac{1}{n^r}A_n(X,2r) - l_P \right| \leq \frac{C}{2^P}
$$
and this being true for any $P$, $C$ being a constant independent from $P$, and since $\displaystyle\lim_{P\rightarrow +\infty}l_P=\frac{V_\nu^r}{r!}$, this concludes the proof that the moments of even order converge towards the moments of the centered normal law and thus, together with Subsection \ref{s.moments_impairs}, proves Theorem \ref{t.principal}.
\flushright $\square$ \flushleft

\bibliographystyle{abbrv}
\bibliography{biblio}

\end{document}